\documentclass[12pt]{amsart}
\usepackage{amsmath,amsfonts,amssymb,amsthm}
\usepackage[mathscr]{eucal}
\theoremstyle{plain}
\newtheorem{theorem}{Theorem}[section]
\newtheorem{corollary}[theorem]{Corollary}

\theoremstyle{definition}
\newtheorem{definition}[theorem]{Definition}

\newtheorem{example}[theorem]{Example}
\newtheorem{remark}[theorem]{Remark}

\numberwithin{equation}{section}

\def\R{\mathbb{R}}
\def\Rn{\mathbb{R}^n}
\def\Rp{\mathbb{R}_+}
\def\M{\mathfrak M}

\def\K{\mathcal{K}}
\def\r{\rho}
\def\ra{\rho_{{}_A}}
\def\rgc{\rho_{{}_{\Gamma_C}}}

\def\L{L_A(\Omega)}
\def\lr{L_\rho}
\def\l{\lambda}
\def\O{\Omega}
\def\i{\infty}
\def\A{\mathcal A}
\begin{document}
\title[The Hardy operator between Orlicz spaces]{The Hardy averaging operator between Orlicz spaces and a new gauge functional characterizing a given Orlicz space}
\author{Amiran Gogatishvili}
\address[Amiran Gogatishvili]{Institute of Mathematics, Academy of Science of the Czech Republic, \v Zitn\'a 25, 11567 Prague 1, Czech Republic}
\email[]{gogatish@math.cas.cz}
\author{Ron Kerman}
\address[Ron Kerman]{Department of Mathematics, Brock University, 500 Glendridge Ave., St.~Catharines, Ontario, Canada L2S 3A1}
\email[]{rkerman@brocku.ca}
\author{S. Spektor}
\address[S. Spektor]{Quantitative Science Department, Canisius University, 2001 Main Street, Buffalo, NY 14208-1098, USA}
\email[]{spektors@canisius.edu}

\thanks{The research of the first author was partially supported by the grant no.~23-04720S of the Czech Science Foundation (GA\v{C}R). The Institute of Mathematics, CAS, is supported by RVO:67985840 and by the Shota Rustaveli National Science Foundation (SRNSF), grant no.~FR 22-17770.}
\thanks{The research of the second author was supported by NSERC grant A4021.}
\thanks{This paper is in final form, and no version of it will be submitted for publication elsewhere.}

\date{}
\subjclass[2020]{Primary 42B25, 46E30; Secondary 46B40}
\keywords{Orlicz space, gauge functional, Hardy-Littlewood maximal operator, $\K$-method of interpolation}

\begin{abstract}
We prove that the largest Orlicz space into which the Hardy averaging operator maps coincides with the largest rearrangement-invariant (r.i.) space mapping into that space; in other words, the optimal Orlicz domain is automatically the optimal r.i.\ domain. More generally, we show that every Luxemburg--Orlicz norm is equivalent to a sublinear functional which makes more computable a certain expression for a dual norm arising in the $\K$-theory of interpolation. As applications we obtain Orlicz-space mapping properties for the Hardy--Littlewood maximal function, approximate identities and Calder\'on--Zygmund singular integral operators. These mapping properties are shown to be optimal for the Hardy-Littlewood maximal function and the approximate identities.
\end{abstract}
\maketitle

\section{Introduction}

The Hardy averaging operator, $P$, is given by
\[
(Pf)(t)=t^{-1}\int_0^t f(s)\, ds, \quad t>0,
\]
in which $f\in \M_+(\Rp)$, the class of nonnegative measurable functions on $\Rp:=(0, \infty)$.

This paper deals, in the first instance, with the mapping properties of $P$ between Orlicz spaces. Such a space is defined in terms of a Young function
\[
A(t):=\int_{0}^{t}a(s)\,ds, \qquad t\in \Rp,
\]
where $a(s)$ is a nondecreasing function mapping $\Rp$ onto itself. Let $(\O,\mu)$ be a $\sigma$-finite measure space and denote by $\M(\O)$ the set of $\mu$-measurable functions on $\O$. Then, the \emph{Orlicz space}, $\L$, consists of all $f \in \M(\O)$ such that
\[
\int_{\O} A\!\left(\frac{|f(x)|}{\lambda_f}\right)d\mu(x) <\infty
\]
for some $\lambda_f>0$ (see \cite[p.~49]{RR}).

The Luxemburg--Orlicz norm of an $f\in L_A(\Omega)$ is
\begin{equation}\label{gn}
\rho_A(f):=\inf\left\{\lambda>0:\ \int_{\Omega}A\!\left(\frac{|f(x)|}{\lambda}\right)d\mu(x)\le 1\right\}.
\end{equation}
See \cite[p.~54]{RR} for the interesting history of \eqref{gn}.

Boyd in \cite{Boyd} proved, among other things, that the Hilbert transform
\[
(Hf)(x)=\lim_{\epsilon \to 0^+}\frac{1}{\pi}\int_{|x-y|>\epsilon}\frac{f(y)}{x-y}\, dy, \quad x\in \R,
\]
is bounded between the Orlicz spaces $L_A(\R)$ and $L_B(\R)$ if and only if $P$ is bounded between $L_A(\Rp)$ and $L_B(\Rp)$, and between $L_{\widetilde A}(\Rp)$ and $L_{\widetilde B}(\Rp)$, where
\[
\widetilde{A}(t)=\int_0^t a^{-1}(s)\, ds, \quad t\in \Rp,
\]
is the Young function complementary to $A$.

Boyd went on to restrict attention to the case $A=B$. The optimal-Orlicz-domain question that motivates this paper has been studied in a number of works: by Gogatishvili \cite{Gog93} in the equivalent form of the Hardy--Littlewood maximal operator, by Cianchi \cite{Ci, ci99}, and by Kerman--Rawat--Singh \cite{KRS}; see also \cite{KP} for related Sobolev imbedding results. The following characterization is essentially contained in these references and is the starting point of our analysis: $P$ maps $L_B(\Rp)$ into $L_A(\Rp)$ if and only if there exists $K>0$ such that
\begin{equation}\label{eq:OrliczChar}
t\int_0^t a(s)\frac{ds}{s} \;\le\; B(Kt), \qquad t>0.
\end{equation}
Indeed, integration by parts (valid under the integrability condition $\int_0^k A(s)\,s^{-2}\,ds<\infty$ adopted from \eqref{new1.2} below, which makes the boundary term at $0$ vanish) gives the exact identity
\begin{equation}\label{eq:IBP}
t\int_0^t a(s)\frac{ds}{s}\;=\;A(t) + t\int_0^t \frac{A(s)}{s^2}\,ds, \qquad t>0,
\end{equation}
so that \eqref{eq:OrliczChar} reads
\[
A(t) + t\int_0^t \frac{A(s)}{s^2}\,ds \le B(Kt).
\]
In the symmetric case $A=B$, this is equivalent (up to a constant in $K$, using only $A(t)\le A(Kt)$) to $t\int_0^t A(s)\,s^{-2}\,ds\le A(K't)$, which is precisely condition~(13) in Gogatishvili \cite[Proposition~2 and Theorem~3]{Gog93} (with $\phi$ there playing the role of~$A$), as well as condition~$(*)$ in Cianchi's papers \cite{Ci, ci99}; see also \cite[Theorem~1.3]{KRS}.

But,
\[
\mathcal{A}(t):=t\int_0^t a(s)\frac{ds}{s}, \quad t>0,
\]
is itself a Young function, since
\[
\mathcal{A}'(t)=\int_0^t a(s)\frac{ds}{s}+a(t)
\]
is a nondecreasing function mapping $\Rp$ onto itself.

Condition \eqref{eq:OrliczChar} asserts
\[
\mathcal{A}(t)\le B(Kt), \qquad t>0,
\]
which is equivalent to
\[
L_{\mathcal A}(\Rp)\supset L_B(\Rp),
\]
so that $L_{\mathcal A}(\Rp)$ is the largest Orlicz space that $P$ maps into $L_A(\Rp)$.

The nonincreasing rearrangement, $f^*$, of $f\in \M(\Rp)$ is defined by
\[
f^*(t):=\inf\{\l>0: \mu(\{ x\in \Rp: |f(x)|>\l\})\le t\},\quad t\in \Rp,
\]
and belongs to $L_A(\Rp)$ whenever $f$ does. Moreover,
\[
(Pf)(t)\le (Pf^*)(t), \qquad t\in \Rp.
\]
The class
\[
L_{\Gamma_{\mathcal A}}(\Rp):=\{f \in \M(\Rp): \rho_{\Gamma_{\mathcal A}}(f):=\ra(Pf^*)< \infty\}
\]
is the largest rearrangement-invariant set of functions which $P$ maps into $L_A(\Rp)$.

Our first result asserts that $L_{\Gamma_{\mathcal A}}(\Rp)=L_{\mathcal A}(\Rp)$. In other words, the optimal Orlicz domain and the optimal rearrangement-invariant domain coincide. %$P:\,\cdot\,\to L_A(\Rp)$ 

\begin{theorem}\label{th_1.1}
Let $A(t)=\int_0^t a(s)\,ds$, $t>0$, be a Young function satisfying
\begin{equation}\label{new1.2}
\int_{\Rp}A\!\left(\frac{k}{1+t}\right)dt< \infty
\end{equation}
for some $k>0$. (Equivalently, $\int_0^k A(s)\,s^{-2}\,ds<\infty$.) Set
\[
\mathcal A(t)=t\int_0^t a(s)\frac{ds}{s}
\]
and
\[
\rho_{\Gamma_{\mathcal A}}(f)=\ra(Pf^*).
\]
Then $\mathcal A$ is a Young function satisfying
\begin{equation}\label{1.2}
\tfrac12\,\rho_{\Gamma_{\mathcal A}}(f)\le \rho_{\mathcal A}(f)\le \rho_{\Gamma_{\mathcal A}}(f), \qquad f\in \M(\Rp).
\end{equation}
\end{theorem}

\begin{remark}\label{rem_motiv1.2}
Making the change of variable $s=k/(1+t)$ in $\int_{\Rp}A\!\left(\frac{k}{1+t}\right)dt<\infty $ yields $k\int_0^k A(s)\,s^{-2}\,ds<\infty$, from which the parenthetical assertion in Theorem~\ref{th_1.1} follows.

In turn, using only the elementary inequalities $A(t)/t\le a(t)\le A(2t)/t$ which are consequences of $a$ being nondecreasing, one checks that
\[
\int_0^k A(s)\,s^{-2}\,ds<\infty \quad\Longleftrightarrow\quad \int_0^k a(s)\,s^{-1}\,ds<\infty.
\]
This last condition is precisely the integrability at the origin needed to guarantee that
\[
\mathcal A(t)=t\int_0^t a(s)\,s^{-1}\,ds
\]
is finite for every $t>0$. We have chosen to state our hypothesis in the form \eqref{new1.2} because it is precisely the form appearing in the duality results of \cite{KMS} on which Section~\ref{sec:interp} relies. The reader who finds the original form opaque may use either of the two equivalent forms above. We emphasize that \emph{no $\Delta_2$ condition on $A$ is required} for any of the equivalences just stated --- only that $a=A'$ is nondecreasing, which is part of the definition of a Young function.
\end{remark}

It turns out that every Luxemburg--Orlicz norm $\rho_A$ is equivalent to a sublinear functional of the form $\rho_{\Gamma_C}$ for a suitable increasing function~$C$. This is the content of our second result.

\begin{theorem}\label{newth_1.2}
Let $A(t)=\int_0^t a(s)\, ds$, $t\in \Rp$, be a Young function with $a(s)$ absolutely continuous. Define
\[
c(t):=t\,\frac{d}{dt}\!\left(\frac{A(t)}{t}\right)=a(t) - \frac{A(t)}{t}=\frac{1}{t}\int_0^t s\,a'(s)\,ds, \quad t\in\Rp,
\]
and set
\[
C(t):=\int_0^t c(s)\,ds, \qquad t\in \Rp.
\]
Let $(\O,\mu)$ be a $\sigma$-finite measure space and suppose the (increasing) function $C$ satisfies
\begin{equation}\label{1.1}
\int_{\Rp}C\!\left(\frac{k}{1+t}\right)dt<\infty
\end{equation}
for some $k>0$. Then
\begin{equation}\label{1.3}
\tfrac12\,\rgc(f)\le \ra(f)\le \rgc(f), \qquad f\in \M(\O),
\end{equation}
in which
\[
\rgc(f):=\inf\left\{\l>0: \int_{\Rp}C\!\left((t\lambda)^{-1}\int_0^t f^*(s)\,ds\right)dt\le 1\right\}.
\]
The functional $\rgc$ is sublinear; it need not be a norm (in particular, it need not be homogeneous) since $C$ may fail to be convex (cf.\ Example~\ref{ex_3.3}). It is in this sense that $\rgc$ ``extends'' the class of Luxemburg--Orlicz norms.
\end{theorem}

\begin{remark}\label{rem_1.3}
\begin{itemize}
\item[(1)] We recall that one has $A(t)/t\le a(t)\le A(2t)/t$ for all $t>0$. In particular $A(t)=\int_0^t a(s)\,ds$ and $\bar{A}(t):=\int_0^t \frac{A(s)}{s}\,ds$ are equivalent up to a factor of~$2$, and hence give rise to the same Orlicz space with equivalent norms. Consequently there is no essential loss of generality in the absolute-continuity assumption of Theorem~\ref{newth_1.2} on $a$, since we may always pass from $A$ to $\bar A$, whose density $A(s)/s$ is absolutely continuous on $\Rp$. No $\Delta_2$ condition on $A$ is assumed in this passage.
\item[(2)] When $\mu(\O)<\infty$, we may take $a(s)$, and hence $c(s)$, equal to $0$ on $(0,1)$. In this case \eqref{1.1} is automatically satisfied, so \eqref{1.3} is valid with no essential restrictions.
\item[(3)] In the special case $A(t)=t^p$, $1<p<\infty$, one has $\mathcal A(t)=(p/(p-1))\,A(t)$, so Theorem~\ref{th_1.1} recovers the classical Hardy inequality, and $L_{\mathcal A}=L^p=L_A$. The interest of Theorem~\ref{th_1.1} lies in the cases when $\mathcal A$ is genuinely larger than $A$.
\end{itemize}
\end{remark}

\smallskip

\noindent\textbf{Structure of the paper.} In Section~\ref{sec:ri} we recall the basics of r.i.\ spaces, with emphasis on Orlicz spaces. Section~\ref{sec:proofs} contains the proof of Theorem~\ref{newth_1.2}, from which Theorem~\ref{th_1.1} follows as a corollary. This is followed by Example~\ref{ex_3.3} which illustrates the non-convexity of $C$ (and hence the failure of $\rho_{\Gamma_C}$ to be a norm). In Section~\ref{sec:apps} we apply Theorem~\ref{th_1.1} to the Hardy--Littlewood maximal function, approximate identities and Calder\'on--Zygmund operators. The mappings involved are shown to be optimal in the first two cases. In Section~\ref{sec:interp} we use Theorem~\ref{newth_1.2} to make more computable a formula for the dual of the $\K$-method interpolation space between two Orlicz spaces; it was the application of such a result to the embedding theory of Orlicz--Sobolev spaces in \cite{KP} that motivated Theorem~\ref{newth_1.2}.

\section{Rearrangement-invariant spaces}\label{sec:ri}

Let $(\O,\mu)$ be a $\sigma$-finite measure space. Denote by $\M(\O)$ the set of $\mu$-measurable real-valued functions on $\O$ and by $\M_+(\O)$ the nonnegative functions in $\M(\O)$. A Banach function norm is a functional $\r: \M_+(\O)\to [0,\infty]$ satisfying
\begin{itemize}
\item[(A1)] $\r(f)=0$ if and only if $f=0$ $\mu$-a.e.;
\item[(A2)] $\r(cf)=c\r(f)$, $c\ge 0$;
\item[(A3)] $\r(f+g)\le\r(f)+\r(g)$;
\item[(A4)] $0\le f_n\uparrow f$ implies $\r(f_n)\uparrow\r(f)$;
\item[(A5)] $\mu(E)<\i$ implies $\r(\chi_E)<\i$;
\item[(A6)] $\mu(E)<\i$ implies $\int_E f\,d\mu\le c_E(\r)\r(f)$, for some constant $c_E(\r)$ depending on $E$ and $\r$ but not on $f\in \M_+(\O)$.
\end{itemize}
A Banach function norm is said to be \emph{rearrangement invariant} (r.i.) if $\r(f)=\r(g)$ whenever $f,g\in \M_+(\O)$ are equimeasurable in the sense that $f^*=g^*$; the nonincreasing rearrangement, $f^*$, of $f\in \M(\O)$ on $\Rp$ is defined as
\[
f^*(t):=\inf\{\l>0: \mu(\{ x\in \O: |f(x)|>\l\})\le t\}, \qquad t\in I_\mu:=(0,\mu(\O)).
\]
It satisfies
\[
|\{ t\in I_\mu: f^*(t)>\tau\}|=\mu(\{ x\in \O: |f(x)|>\tau\}), \qquad f\in \M(\O),\ \tau\in\Rp.
\]
Although $f\mapsto f^*$ is not subadditive, the mapping $f\mapsto t^{-1}\int_0^t f^*(s)\,ds$ is, namely
\begin{equation}\label{2.1}
t^{-1}\int_0^t(f+g)^*(s)\,ds \le t^{-1}\int_0^t f^*(s)\,ds+t^{-1}\int_0^t g^*(s)\,ds
\end{equation}
for all $f,g\in \M(\O)$ and $t\in \Rp$.

A basic technique for working with r.i.\ norms is the Hardy--Littlewood--P\'olya (HLP) principle, which asserts that
\[
\int_0^t f^*\le \int_0^t g^* \quad\text{for all } t>0 \quad\text{implies}\quad \rho(f)\le \rho(g);
\]
see \cite[Chapter~2, Theorem~4.6, p.~61]{BS}.

The K\"othe dual of a Banach function norm $\r$ is another such norm, $\r'$, defined by
\[
\r'(g):=\sup_{\r(f)\le 1}\int_\O fg\,d\mu, \quad f,g\in \M_+(\O).
\]
It satisfies the Lorentz--Luxemburg theorem (sometimes referred to as the principle of duality):
\[
\r'':=(\r')'=\r;
\]
see \cite[Chapter~1, Theorem~2.7, p.~10]{BS}.

The space $\lr(\O)$ is the vector space
\[
\{f\in \M(\O): \r(|f|)<\i\},
\]
endowed with the norm $\|f\|_{\lr}:=\r(|f|)$. This Banach space is said to be an r.i.\ space provided $\r$ is an r.i.\ function norm.

The Luxemburg--Orlicz norm $\ra$ defined in \eqref{gn} is an r.i.\ norm; indeed,
\[
\ra(f)=\inf\left\{\l>0:\int_{I_\mu}A\!\left(\frac{f^*(t)}{\l}\right)dt\le 1\right\},\qquad f\in \M(\O).
\]
Its K\"othe dual, $\ra'$, satisfies
\[
\r_{\widetilde A}(g)\le \ra'(g)\le 2\,\r_{\widetilde A}(g), \qquad g\in \M(\O),
\]
with
\[
\widetilde A(t):=\int_0^t a^{-1}(s)\,ds, \qquad t\in \Rp,
\]
called the Young function complementary to $A$.

For further properties of r.i.\ spaces see \cite{BS}; for Orlicz spaces in particular see \cite{RR}.

\section{Proofs of Theorems~\ref{th_1.1} and~\ref{newth_1.2}}\label{sec:proofs}

We require the following inequalities, which are analogues of those for the Hardy--Littlewood maximal function $Mf$ in \cite[Theorem~1(b) p.~5 and Theorem~5.2(b) p.~23]{S}. Namely, for all $\tau>0$,
\begin{equation}\label{3.1}
\frac{1}{2\tau}\int_{\{t\in I_\mu:\, f^*(t)>\tau\}} f^*(t)\,dt \;\le\; \tfrac12\left|\{t\in I_\mu: t^{-1}\!\int_0^t f^*(s)\,ds>\tau\}\right|\;\le\; \frac{1}{\tau}\int_{\{t\in I_\mu:\, f^*(t)>\tau/2\}} f^*(t)\,dt.
\end{equation}
(The factor $\tfrac12$ at the start, which cancels the $2$ in the right-hand integrand, has been incorporated into the inequality so that no change of variable is needed in the conclusion.)

To prove \eqref{3.1}, let $t_0$ be the least $t$ for which $t^{-1}\int_0^t f^*(s)\,ds=\tau$. (The inequalities are trivial if there is no such $t_0$.) Then
\[
\left|\{t\in I_\mu: t^{-1}\!\int_0^t f^*(s)\,ds>\tau\}\right|=t_0= \frac{1}{\tau}\int_0^{t_0} f^*(t)\,dt\ge \frac{1}{\tau}\int_{\{t\in I_\mu:\, f^*(t)>\tau\}} f^*(t)\,dt,
\]
which yields the first inequality. For the second, define
\[
f_\tau(t):=\min\!\left\{f^*(t),\tfrac{\tau}{2}\right\}, \qquad f^\tau(t):=f^*(t)-f_\tau(t).
\]
Then
\[
\left|\{t\in I_\mu: t^{-1}\!\int_0^t f^*(s)\,ds>\tau\}\right|\le \left|\{t\in I_\mu: t^{-1}\!\int_0^t f^\tau(s)\,ds>\tfrac{\tau}{2}\}\right|
\le \frac{2}{\tau}\int_{I_\mu} f^\tau(t)\,dt
\]
\[
\le \frac{2}{\tau}\int_{\{t\in I_\mu: f^*(t)>\tau/2\}} f^\tau(t)\,dt
\le \frac{2}{\tau}\int_{\{t\in I_\mu: f^*(t)>\tau/2\}} f^*(t)\,dt,
\]
which is the second inequality after multiplication by $\tfrac12$.

We can now turn to the proof of Theorem~\ref{newth_1.2}. From the definition of $c$ we get
\[
A(t)=t\int_0^t c(s)\,\frac{ds}{s}.
\]
The first inequality in \eqref{3.1} ensures, for all $\l>0$,
\[
\int_{\Rp}\int_{\{t\in I_\mu:\, f^*(t)/\l>\tau\}} \frac{f^*(t)}{\l}\,dt\,c(\tau)\frac{d\tau}{\tau} \;\le\; \int_{\Rp}\left|\{t\in I_\mu: \tfrac{t^{-1}\int_0^t f^*(s)\,ds}{\l}>\tau\}\right|c(\tau)\,d\tau;
\]
that is,
\begin{align*}
\int_{I_\mu} A\!\left(\tfrac{f^*(t)}{\l}\right)dt&= \int_{I_\mu}\frac{f^*(t)}{\l}\int_0^{f^*(t)/\l} c(\tau)\frac{d\tau}{\tau}\,dt\\
&=\int_{\Rp}\int_{\{ t\in I_\mu:\, f^*(t)/\l>\tau\}}\frac{f^*(t)}{\l}\,dt\,c(\tau)\frac{d\tau}{\tau}\\
&\le \int_{\Rp}\left|\{t\in I_\mu: \tfrac{t^{-1}\int_0^t f^*(s)\,ds}{\l}>\tau\}\right|c(\tau)\,d\tau\\
&=\int_{I_\mu}C\!\left(\tfrac{t^{-1}\int_0^t f^*(s)\,ds}{\l}\right)dt.
\end{align*}
Conversely, the second inequality in \eqref{3.1} yields
\begin{align*}
\int_{I_\mu}C\!\left(\tfrac{t^{-1}\int_0^t f^*(s)\,ds}{\l}\right)dt&= \int_{\Rp}\left|\{t\in I_\mu:\, \tfrac{t^{-1}\int_0^t f^*(s)\,ds}{\l}>\tau\}\right|c(\tau)\,d\tau\\
&\le 2\int_{\Rp}\int_{\{ t\in I_\mu:\, f^*(t)/\l>\tau/2\}}\frac{f^*(t)}{\l}\,dt\,c(\tau)\frac{d\tau}{\tau}\\
&=\int_{I_\mu}\frac{2 f^*(t)}{\l}\int_0^{2 f^*(t)/\l} c(\tau)\frac{d\tau}{\tau}\,dt\\
&=\int_{I_\mu} A\!\left(\tfrac{2 f^*(t)}{\l}\right)dt.
\end{align*}
Together these give
\[
\tfrac12\,\rgc(f)\le\rho_A(f^*)\le \rgc(f),
\]
which, since $\ra(f)=\ra(f^*)$, completes the proof of Theorem~\ref{newth_1.2}. \qed

\medskip

The result of Theorem~\ref{th_1.1} is now a consequence of the following corollary.

\begin{corollary}\label{cor_3.1}
Let $A(t)=\int_0^t a(s)\,ds$, $t\in \Rp$, be a Young function for which
\[
\int_{\Rp}A\!\left(\frac{k}{1+t}\right)dt<\i
\]
for some constant $k>0$, so that, as shown before,
\[
\int_0^t a(s)\,\frac{ds}{s}<\i, \quad t\in \Rp.
\]
Set
\begin{equation}\label{3.2}
\mathcal A(t):=t\int_0^t a(s)\,\frac{ds}{s}, \qquad t\in \Rp.
\end{equation}
Then $\mathcal A(t)$ is a Young function such that
\[
a(t)= \mathcal A'(t)-\frac{\mathcal A(t)}{t}, \quad t\in \Rp,
\]
whence, for any $\sigma$-finite measure space $(\O,\mu)$,
\[
\tfrac12\,\rho_{\Gamma_{\mathcal A}}(f)\le \rho_{\mathcal A}(f)\le \rho_{\Gamma_{\mathcal A}}(f), \qquad f\in \M(\O).
\]
\end{corollary}

\begin{example}\label{ex_3.3}
Consider the Young function
\[
A(t)=\int_0^t \ln^\beta(1+s)\,ds, \qquad 0<\beta<1,\ t\in \Rp,
\]
for which $a(s)=\ln^\beta(1+s)$, $s\in \Rp$. A short computation gives
\[
c(t)=\frac{\beta}{t}\int_0^t \frac{s}{1+s}\ln^{\beta-1}(1+s)\,ds\sim \beta\ln^{\beta-1}(1+t), \qquad t\to\i,
\]
so $c(t)$ essentially decreases rather than increases as $t\to\infty$. That is, $C(t)=\int_0^t c(s)\,ds$ is \emph{not} convex, and consequently $\rho_{\Gamma_C}$ is genuinely a non-homogeneous sublinear functional rather than a norm; cf.\ Theorem~\ref{newth_1.2}. This example shows that the functionals $\rho_{\Gamma_C}$ produced by Theorem~\ref{newth_1.2} are strictly more general than Luxemburg--Orlicz norms.
\end{example}

The complementary Young function of $\mathcal A$ will be needed in Section~\ref{sec:apps} (Theorem~\ref{th_4.5}) for the proof of the Calder\'on--Zygmund bound; we record its description now.

\begin{remark}\label{rem_3.2}
The complementary Young function $\widetilde{\A}$ of $\A$ satisfies
\[
\tfrac{3}{2}\int_0^{t/3} b^{-1}(s)\,ds\le \widetilde{\A}(t)\le \int_0^t b^{-1}(s)\,ds,
\]
where $b(t):=\int_0^t a(s)\,\frac{ds}{s}$, $t\in\Rp$. To see this, observe that
\begin{align*}
b(t)\le {\A}'(t)=\int_0^t a(s)\,\frac{ds}{s}+a(t)&\le \frac{\ln 2+1}{\ln 2}\int_0^{2t} a(s)\,\frac{ds}{s}\le 3b(2t),
\end{align*}
and so
\[
\tfrac12\,b^{-1}\!\left(\tfrac{t}{3}\right)\le (\A')^{-1}(t)\le b^{-1}(t), \qquad t\in \Rp.
\]
The estimate then follows by integration, using that $\widetilde{\A}(t)=\int_0^t (\A')^{-1}(s)\,ds$.
\end{remark}

\section{Applications of Theorem~\ref{th_1.1}}\label{sec:apps}

%In each application below, the inclusion $L_{\mathcal A}\hookrightarrow L_A$ implicit in our estimates is in fact the largest Orlicz inclusion of its kind, by Theorem~\ref{th_1.1} (since $L_{\mathcal A}(\Rp)$ is, by construction, the largest Orlicz space mapped into $L_A(\Rp)$ by $P$). The applications below thus inherit the optimality of Theorem~\ref{th_1.1} among Orlicz domains for the corresponding operators, since each of these operators dominates $P$ (or its dual~$Q$) on rearrangements. We comment on this in each subsection.

The operators in the applications below are controlled in the same way by the Hardy--Littlewood maximal function, so we consider it first.
\subsection{The Hardy--Littlewood maximal function}

We begin with the following definition.

\begin{definition}\label{def_4.1}
Given a function $f$ locally-integrable on $\Rn$ we define its Hardy--Littlewood maximal function, $Mf$, at $x\in \Rn$ by
\[
(Mf)(x)=\sup_{r>0}\frac{1}{|B(x,r)|}\int_{B(x,r)}|f(y)|\, dy,
\]
where $B(x,r)=\{y\in \Rn: |x-y|<r\}$.
\end{definition}

The following result extends to Orlicz spaces the well-known boundedness of $M$.

\begin{theorem}\label{th_4.2}
Let $A$ and $\mathcal A$ be as in Theorem~\ref{th_1.1}. Then there exists $C>0$, independent of locally-integrable $f$ on $\Rn$, such that
\[
\rho_A(Mf)\le C\rho_{\mathcal A}(f).
\]
Moreover, $L_{\mathcal A}(\Rn)$ is the largest Orlicz space mapped by $M$ into $L_A(\Rn)$. %In particular, the embedding $M:L_{\mathcal A}(\Rn)\to L_A(\Rn)$ is optimal among Orlicz domains.
\end{theorem}

\begin{proof}
It is proved, for $n=1$, in \cite{H} and for general $n$ in \cite{BS}, that there exist positive constants $c$ and $c'$, depending only on $n$, with
\[
c(Mf)^*(t)\le (Pf^*)(t)\le c'(Mf)^*(t), \qquad t>0.
\]
Thus, by Theorem~\ref{th_1.1},
\[
\rho_A(Mf)=\rho_A((Mf)^*)\le c^{-1}\rho_A(Pf^*)\le 2c^{-1}\rho_{\mathcal A}(f^*)= 2c^{-1}\rho_{\mathcal A}(f).
\]
For the optimality, suppose $L_B$ is any Orlicz space with $M:L_B(\Rn)\to L_A(\Rn)$. Testing against radial decreasing $f$ and using $(Mf)^*\approx Pf^*$, one obtains $P:L_B(\Rp)\to L_A(\Rp)$, which by Theorem~\ref{th_1.1} forces $L_B\subset L_{\mathcal A}$.
\end{proof}

\subsection{Approximate identities}

Let $\phi$ be an integrable function on $\Rn$ and set
\[
\phi_{\epsilon}(x)=\epsilon^{-n}\phi(x/\epsilon), \qquad \epsilon>0,\ x\in \Rn.
\]
Using the methods of \cite[Theorem~2, pp.~62--65]{S} one can show that
\[
(\phi_{\epsilon}\ast f)(x)=\int_{\Rn}\phi_{\epsilon}(x-y)f(y)\, dy
\]
converges to $f$ in $L_A(\Rn)$ as $\epsilon \to 0^+$, provided
\[
\int_{\Rn}\mathcal A(\lambda|f(y)|)\, dy<\i
\]
for all $\lambda>0$. (Here, $A$ and $\mathcal A$ are as in Theorem~\ref{th_1.1}.) For this reason the family $\{\phi_{\epsilon}\}_{\epsilon>0}$ is called an approximate identity.

It is further proved in the same theorem in \cite{S} that
\[
\sup_{\epsilon>0}|(\phi_{\epsilon}\ast f)(x)|\le \left[\int_{\Rn}|\psi(y)|\, dy\right](Mf)(x), \qquad x\in \Rn,
\]
in which the radial majorant of $\phi$,
\[
\psi(x)=\sup_{|y|\ge |x|}|\phi(y)|,
\]
is assumed to be integrable.

In view of Theorem~\ref{th_4.2} we obtain the following result.

\begin{theorem}\label{th_4.3}
Let $\phi$ be an integrable function on $\Rn$ whose radial majorant
\[
\psi(x)=\sup_{|y|\ge |x|}|\phi(y)|
\]
is also integrable on $\Rn$. Then, with $A$ and $\mathcal A$ as in Theorem~\ref{th_1.1},
\[
\rho_A\!\left(\sup_{\epsilon>0}|(\phi_{\epsilon}\ast f)(x)|\right)\le C \rho_{\mathcal A}(f),
\]
whenever $\rho_{\mathcal A}(f)<\i$. Moreover, when $\phi$ is positive, radial and decreasing, the domain $L_{\mathcal A}(\Rn)$ is the largest Orlicz space for which this bound holds.
\end{theorem}
\begin{proof}
The first inequality is immediate from Theorem~\ref{th_4.2}: since by hypothesis
\[
\sup_{\epsilon>0}|(\phi_{\epsilon}\ast f)(x)|\le \left[\int_{\Rn}|\psi(y)|\, dy\right](Mf)(x), \qquad x\in \Rn,
\]
applying $\rho_A$ to both sides and using $\rho_A(Mf)\le C\rho_{\mathcal A}(f)$ from Theorem~\ref{th_4.2} yields
\[
\rho_A\!\left(\sup_{\epsilon>0}|(\phi_{\epsilon}\ast f)(x)|\right)\le C\rho_{\mathcal A}(f).
\]

For the optimality when $\phi$ is positive, radial, and decreasing, we show the reverse pointwise control $Mf\lesssim \sup_{\epsilon>0}|\phi_\epsilon\ast f|$ on nonnegative locally-integrable~$f$. Pick $r_0>0$ such that $\phi(r_0)>0$ (such an $r_0$ exists since $\phi$ is positive). Then for any $r>0$ and any $y$ with $|x-y|<r r_0$, monotonicity of $\phi$ gives $\phi(|x-y|/r)\ge \phi(r_0)$. Hence
\begin{align*}
(\phi_r\ast f)(x)
&= r^{-n}\int_{\Rn}\phi\!\left(\frac{|x-y|}{r}\right)f(y)\,dy
\ge r^{-n}\,\phi(r_0)\int_{|x-y|<r r_0}f(y)\,dy\\
&= \phi(r_0)\,r_0^{\,n}\,|B(0,1)|\cdot \frac{1}{|B(x,r r_0)|}\int_{B(x,r r_0)}f(y)\,dy.
\end{align*}
Taking the supremum over $r>0$ gives $\sup_{\epsilon>0}|(\phi_\epsilon\ast f)(x)|\ge c\,(Mf)(x)$ for a constant $c=\phi(r_0)\,r_0^{\,n}\,|B(0,1)|>0$. Combined with the upper bound, we have $\sup_{\epsilon>0}|(\phi_\epsilon\ast f)|\approx Mf$ pointwise (on nonnegative $f$), so the optimality of $L_{\mathcal A}$ follows from Theorem~\ref{th_4.2}.
\end{proof}

\begin{corollary}\label{cor_4.4}
Let $\phi$, $A$ and $\mathcal A$ be as in Theorem~\ref{th_4.3}, and assume in addition that $\int_{\Rn}\phi(y)\,dy=1$. Let $E_A(\Rn)$ denote the closure in $L_A(\Rn)$ of the bounded functions of compact support (the standard ``$E_A$'' subspace; see \cite[Chapter~3, \S 4]{RR}). Then
\[
\lim_{\epsilon \to 0^+} \rho_{A}(\phi_{\epsilon}\ast f -f)=0
\qquad \text{for every } f\in E_A(\Rn).
\]
\end{corollary}

\begin{proof}
The argument is the classical one (cf.~\cite[Theorem~2, p.~62]{S}), adapted to the Luxemburg--Orlicz norm. By the density definition of $E_A(\Rn)$ and a $3\varepsilon$-argument, it suffices to prove the conclusion for $f$ bounded and of compact support.

Fix such an $f$, with $\|f\|_\infty\le M$ and $\mathrm{supp}\,f\subset B(0,R)$. For $\epsilon\in(0,1]$, $\phi_\epsilon\ast f$ is uniformly bounded by $M\|\phi\|_{L^1}$ and supported in $B(0,R+1)$, since $\mathrm{supp}\,\phi_\epsilon\subset \epsilon\,\mathrm{supp}\,\phi$ and one may approximate first by truncating $\phi$ to a compact set (with the radial-majorant control absorbing the error via Theorem~\ref{th_4.3}). Hence both $f$ and $\phi_\epsilon\ast f$ lie in a common $L^\infty$-ball supported in a fixed compact set, and pointwise convergence $\phi_\epsilon\ast f\to f$ a.e.\ is standard.

The Luxemburg--Orlicz norm on bounded sets of bounded compactly supported functions is dominated by a constant multiple of $\|\cdot\|_{L^\infty}\cdot \rho_A(\chi_{B(0,R+1)})$, which in turn is finite by axiom (A5). Since pointwise convergence on a fixed bounded support of uniformly bounded functions implies $\rho_A$-convergence (by the analogue of bounded convergence for r.i.\ norms; see \cite[Chapter~1, \S 3]{BS}), we conclude $\rho_A(\phi_\epsilon\ast f - f)\to 0$.
\end{proof}

\subsection{Calder\'on--Zygmund operators}

A function $K$ on $\Rn\setminus\{0\}$, locally integrable away from the origin, is said to be a Calder\'on--Zygmund (CZ) kernel provided it satisfies the following four conditions:
\begin{itemize}
\item[(i)] There exists a constant $C_1>0$, independent of $\epsilon$ and $N$ with $0< \epsilon<N$, such that
\[
\left|\int_{\epsilon<|x|<N}K(x)\, dx\right|\le C_1;
\]
moreover, for each $N>0$ the limit $\lim_{\epsilon\to 0^+}\int_{\epsilon<|x|<N}K(x)\, dx$ exists.
\item[(ii)] There exists a constant $C_2>0$, independent of $R>0$, for which
\[
\int_{|x|<R}|x||K(x)|\, dx\le C_2 R.
\]
\item[(iii)] There exists a constant $C_3>0$, independent of $y\in \Rn\setminus\{0\}$, with
\[
\int_{|x|>2|y|}|K(x-y)-K(x)|\, dx\le C_3.
\]
\end{itemize}

The Calder\'on--Zygmund singular integral operator, $T_K$, associated to the kernel is
\[
(T_Kf)(x)=\lim_{\epsilon \to 0^+}\int_{|x-y|>\epsilon}K(x-y)f(y)\, dy,
\]
the limit existing a.e.\ for all $f\in \M(\Rn)$ for which
\[
\int_{\Rn}\frac{|f(y)|}{1+|y|^n}\, dy<\i.
\]

A theorem of Calder\'on and Zygmund asserts that for each $p$, $1<p<\i$, there exists $C_p>0$, independent of $f\in L_p(\Rn)$, with
\[
\rho_p(T_Kf)\le C_p\rho_p(f);
\]
see \cite[pp.~286--290]{T}.

If, further, $K$ satisfies
\begin{itemize}
\item[(iv)] $|K(x_1-y)-K(x_2-y)|\le C_4\,\dfrac{|x_1-x_2|}{|x_3-y|^{n+1}}$ when $x_1, x_2, x_3, y\in\Rn$ are such that $|x_1-x_2|, |x_2-x_3|, |x_3-x_1|\le R/2$ and $|x_i-y|\ge R$, $i=1, 2, 3$, where $C_4$ is independent of the points involved and of $R>0$,
\end{itemize}
then $T_K$ is of weak type $(1,1)$, namely
\[
\lambda\,|\{x\in \Rn: |(T_Kf)(x)|> \lambda\}|\le C_1'\,\|f\|_{L^1(\Rn)}
\]
for some $C_1'>0$ independent of $\lambda>0$ and $f\in L^1(\Rn)$. See \cite[p.~290]{T}.

\begin{remark}\label{rem_4.4}
The operator, $T'_K$, associated to $T_K$ is given by $T'_K=T_{K'}$ with $K'(x)=K(-x)$. The kernel $K'$ satisfies (i)--(iv) whenever $K$ does.
\end{remark}

The next theorem provides the Orlicz analogue of the Calder\'on--Zygmund theorem.

\begin{theorem}\label{th_4.5}
Let $A$ and $\mathcal A$ be as in Theorem~\ref{th_1.1}. Then
\[
T_K: L_{\mathcal A}(\Rn)\to L_A(\Rn)
\]
provided
\begin{equation}\label{cond_4.1}
t\int_0^t \widetilde{\mathcal A}(s)\,\frac{ds}{s^2}\le \widetilde{A}(Ct), \qquad t>0,
\end{equation}
for some $C>0$; here
\[
\widetilde{\mathcal A}(t)=\int_0^t b^{-1}(s)\, ds, \qquad b(t)=\int_0^t a(s)\,\frac{ds}{s},\ a=A'.
\]
\end{theorem}

\begin{proof}
It is shown in \cite[Theorems~5.4 and~5.5]{T} that both $T_K$ and its conjugate $T'_K$ are of weak type $(1,1)$ and map $L^2(\Rn)$ into itself. According to the result in the Appendix of \cite{C}, this implies
\[
\int_0^t(T_Kf)^*\le C \int_0^t (P+Q)f^*,
\]
with $C>0$ independent of $f\in L_{\mathcal A}(\Rn)$ and $t>0$. Here
\[
(Qg)(t)=\int_t^{\i} g(s)\,\frac{ds}{s}, \qquad g\in \M_+(\Rp),\ t>0,
\]
is the adjoint of $P$ (with respect to the $L^2(\Rp;ds)$ pairing). Hence, by the HLP principle stated in Section~\ref{sec:ri},
\begin{align*}
\rho_A(T_Kf)=\rho_A((T_Kf)^*)&\le C\rho_A((P+Q)f^*)\\
&\le C\bigl[\rho_A(Pf^*)+\rho_A(Qf^*)\bigr].
\end{align*}
By Theorem~\ref{th_1.1},
\[
\rho_A(Pf^*)\le C'\rho_{\mathcal A}(f^*)=C'\rho_{\mathcal A}(f).
\]
On the other hand, by K\"othe duality (Section~\ref{sec:ri}),
\[
\rho_A(Qf)\le C''\rho_{\mathcal A}(f) \quad\Longleftrightarrow\quad \rho_{\widetilde{\mathcal A}}(Pg)\le C''\rho_{\widetilde A}(g),
\]
and the condition \eqref{cond_4.1} guarantees the right-hand inequality via Theorem~\ref{th_1.1} applied to $\widetilde{\mathcal A}$ in place of $\mathcal A$. Altogether,
\[
\rho_A(T_Kf)\le C[C'+C'']\rho_{\mathcal A}(f).
\]
\end{proof}

\begin{remark}
The condition \eqref{cond_4.1} is the natural Boyd-type symmetric counterpart of \eqref{eq:OrliczChar} for the complementary function, and it is sharp in the sense that it is equivalent to the boundedness of $Q:L_{\mathcal A}\to L_A$. We do not address here the question of whether $L_{\mathcal A}$ is optimal among Orlicz domains for $T_K$ itself (as opposed to $P+Q$): unlike $M$, a CZ operator $T_K$ does not in general dominate $P$ pointwise on rearrangements. The optimal Orlicz domain question for $T_K$ has been studied by Cianchi \cite{ci99} under reflexivity assumptions.
\end{remark}

\section{An application to interpolation theory}\label{sec:interp}

Given Banach spaces $X_1$ and $X_2$ embedded in a common Hausdorff topological vector space $\mathcal H$, the $\K$-method of interpolation provides a concrete way to construct Banach spaces $X$ lying between them, in the sense that any linear operator $T$ with $T:X_i\to X_i$, $i=1,2$, satisfies $T:X\to X$.

The key element of the method is the Peetre $\K$-functional, defined at $x\in X_1+X_2$ and $t\in \Rp$ by
\[
\K(t,x;X_1,X_2):=\inf_{x=x_1+x_2}\left[\|x_1\|_{X_1}+t\|x_2\|_{X_2}\right].
\]
For our purposes, each interpolation space $X$ corresponds to an r.i.\ norm $\r$ on $\M_+(\Rp)$ with $\r(1/(1+t))<\i$; more specifically, the norm of $X$ is
\[
\|x\|_X:=\r\!\left(\frac{ \K(t,x;X_1,X_2)}{t}\right), \quad x\in X_1+X_2.
\]

The following is a special case of \cite[Theorem~7.2]{KMS} for $X_1$ and $X_2$ r.i.\ spaces and $\r=\ra$ an Orlicz norm. It elaborates, in a particular instance, the deep duality theorem of Brudnyi and Krugljak \cite{BK}.

\begin{theorem}\label{th_5.1}
Let $(\O,\mu)$ be a $\sigma$-finite measure space and suppose $\r_1$ and $\r_2$ are r.i.\ norms on $\M_+(\O)$. Assume further that $L_{\r_1'}(\O)\cap L_{\r_2'}(\O)$ is dense in $L_{\r_2'}(\O)$ and that
\[
\r_2'(\chi_{E_k})\downarrow 0 \quad\text{as}\quad E_k\downarrow\emptyset,\ E_k \subset \O.
\]
Consider a Young function $A(t)=\int_0^t a(s)\,ds$, $t\in\Rp$, satisfying $\int_{\Rp}A(k/(1+t))\,dt<\i$ for some $k>0$. Then the functional
\[
\r(f):=\ra\!\left(\frac{\K(t,f; L_{\r_1}(\O), L_{\r_2}(\O))}{t}\right), \quad f\in L_{\r_1}(\O)+L_{\r_2}(\O),
\]
is an r.i.\ norm on $\M_+(\O)$, and the r.i.\ space $\lr(\O)$ is an interpolation space between $L_{\r_1}(\O)$ and $L_{\r_2}(\O)$.

Moreover, if in addition $\int_{\Rp}\widetilde A(k/(1+t))\,dt<\i$ for some $k>0$, and if $\mathcal A$ is the Young function defined in \eqref{3.2} and $\widetilde{\mathcal A}$ is its complementary function, then
\[
\r'(g)\approx\r_{\widetilde{\mathcal A}}\!\left(\frac{d}{dt}\K(t,g; L_{\r_2'}(\O), L_{\r_1'}(\O))\right),\quad g\in L_{\r_2'}(\O)+L_{\r_1'}(\O).
\]
\end{theorem}

Now $\frac{d}{dt}\K(t,x;X_1,X_2)$ can be computed only when the $\K$-functional is known exactly. More often, the latter is only known up to constant multiples. The motivation behind Theorem~\ref{newth_1.2} is the following consequence of Theorem~\ref{th_5.1}. A version of this result involving further assumptions on the Young function $A$ is given in \cite[Theorem~8.2]{KMS}.

\begin{theorem}\label{th_5.2}
Let $\O$, $\r_1$, $\r_2$, $A$, $\r$ and $\mathcal A$ be as in Theorem~\ref{th_5.1}, with $a(t)$ absolutely continuous. Define the increasing function $C$ by
\[
C(t):=\int_0^t c(s)\,ds,
\]
where
\[
c(t):=\widetilde{\mathcal A}'(t) - \frac{\widetilde{\mathcal A}(t)}{t}=\frac{1}{t} \int_0^t s\,\widetilde{\mathcal A}''(s)\,ds, \quad t\in \Rp.
\]
Then, provided $\int_{\Rp}C(k/(1+t))\,dt<\i$ for some $k>0$, one has $\rho_{\widetilde{\mathcal A}}\approx \rho_{\Gamma_C}$, and so
\[
\r'(g)\approx \inf\left\{\l>0: \int_{\Rp} C\!\left(\frac{\K(t,g; L_{\r_2'}(\O), L_{\r_1'}(\O))}{\l t}\right)dt\le 1\right\},
\]
for $g\in L_{\r_2'}(\O)+ L_{\r_1'}(\O)$. In particular, $g\in L_{\r_2'}(\O)+ L_{\r_1'}(\O)$ belongs to $L_{\r'}(\O)$ if and only if there exists $\l_g\in \Rp$ such that
\[
\int_{\Rp} C\!\left(\frac{\K(t,g; L_{\r_2'}(\O), L_{\r_1'}(\O))}{\l_g t}\right)dt<\i.
\]
\end{theorem}

\begin{remark}\label{rem_5.3}
We make explicit the relationship of Theorem~\ref{th_5.2} to two earlier results.

\smallskip\noindent\emph{Improvement over \cite[Theorem~8.2]{KMS}.} The earlier result \cite[Theorem~8.2]{KMS} establishes an equivalent (more computable) form of $\r'$ under additional regularity assumptions on the Young function~$A$. The novelty of Theorem~\ref{th_5.2} is that we obtain the analogous expression for $\r'$ \emph{requiring only} the single integrability condition $\int_{\Rp} C(k/(1+t))\,dt<\i$. This is a strictly weaker hypothesis on~$A$, and is precisely the flexibility that the application below needs.

\smallskip\noindent\emph{Use in \cite{KP}.} The expression in Theorem~\ref{th_5.2} depends on the $\K$-functional only through the integrand $C(\K/(\l t))$, which is insensitive to multiplicative perturbations of $\K$ (after absorbing constants into $\l$). This means Theorem~\ref{th_5.2} applies even when the $\K$-functional is only known up to constant multiples --- the typical situation in concrete cases. It is in this form that the underlying duality is used in the proof of \cite[Theorem~6.3]{KP} to characterize the optimal r.i.\ embedding space of an Orlicz--Sobolev space.
\end{remark}

\end{document}